\documentclass{article}
\usepackage{amsmath}
\usepackage{amsthm}
\usepackage{amsfonts}
\usepackage{amssymb}
\usepackage[all]{xy}
\newtheorem{thm}{Theorem}[section]
\newtheorem*{mthm}{Main Theorem}
\newtheorem{prop}[thm]{Proposition}
\newtheorem{cor}[thm]{Corollary}
\newtheorem{lem}[thm]{Lemma}
\theoremstyle{definition}
\newtheorem{defin}[thm]{Definition}
\newtheorem*{notat}{Notation}
\newtheorem*{acknow}{Acknowledgements}
\newtheorem{rem}[thm]{Remark}
\newtheorem{exe}[thm]{Example}

\DeclareMathOperator{\dfl}{d}
\DeclareMathOperator{\im}{Im}

\DeclareMathOperator{\rHom}{RHom}
\DeclareMathOperator{\rHomu}{R\underline{Hom}}
\DeclareMathOperator{\Ker}{Ker}
\DeclareMathOperator{\rp}{R\pi}
\DeclareMathOperator{\rd}{R}
\DeclareMathOperator{\lm}{L}
\DeclareMathOperator{\li}{L\it{i}}
\DeclareMathOperator{\lp}{L\pi}
\DeclareMathOperator{\id}{id}
\DeclareMathOperator{\coh}{Coh}
\DeclareMathOperator{\Ho}{H}
\DeclareMathOperator{\rgam}{R\Gamma}
\DeclareMathOperator{\Ext}{Ext}
\DeclareMathOperator{\cl}{Cl}
\DeclareMathOperator{\supp}{Supp}
\DeclareMathOperator{\holi}{holim}

\DeclareMathOperator{\cone}{Cone}

\DeclareMathOperator{\bl}{Bl}
\DeclareMathOperator{\tot}{Tot}
\DeclareMathOperator{\Der}{D}
\DeclareMathOperator{\rga}{R\Gamma}
\renewcommand{\ker}{\Ker}
\renewcommand{\tilde}{\widetilde}

\newcommand{\qc}{\mathrm{qc}}
\newcommand{\bd}{\mathrm{b}}
\newcommand{\co}{\mathrm{coh}}
\newcommand{\rpi}{\rp _*}
\newcommand{\rip}{\rd ^i \pi _*}
\newcommand{\rkp}{\rd ^{i - (k - 2)} \pi _*}
\newcommand{\ripm}{\rd ^{i - 1} \pi _*}
\newcommand{\lpi}{\lp ^*}
\newcommand{\hoi}{\Ho ^i}
\newcommand{\holim}{\mathop{\holi}}
\renewcommand{\subseteq}{\subset}
\title{Semiorthogonal decompositions of projective varieties with isolated rational singularities}
\author{Yuto Arai}
\date{}
\begin{document}
 \maketitle
 \begin{abstract}
  We develop the method of inducing semiorthogonal decompositions of projective varieties with isolated rational singularities from those of small resolutions of singularities, which generalizes semiorthogonal decompositions for singular surfaces by Karmazyu-Kuznetsov-Shinder. 
  \par We first explain the classical generator of the null category, which is a kind of triangulated full subcategory, and prove that the orthogonal decompositions of the null category is induced through its generator. Next, we prove that the candidate of induced semiorthogonal decomposition behaves well with respect to the inverse image of resolution morphism, and as a corollary, we obtain the required semiorthogonal decomposition.
 \end{abstract}
 \tableofcontents
 \section{Introduction}
 In this paper, we study about semiorthogonal decompositions of the derived categories of projective varieties with isolated rational singularities. It is important to investigate semiorthogonal decompositions in understanding derived categories. There are many methods of creating semiorthogonal decompositions of smooth varieties. For instance, we can induce the semiorthogonal decompositions from the semiorthogonal decompositions of derived categories of projective spaces $\mathbb{P} ^n$, which is well known by Beilinson theorem, by using blowing-up formula and projective bundle formula (the latter formula is generalized to fibrations with smooth proper fibres by \cite[Theorem 2.8]{MR2802576}). However, there are few methods of creating semiorthogonal decompositions of singular varieties. In \cite{MR4276320,MR4382477}, Karmazyn, Kuznetsov, and Shinder study the methods of inducing semiorthogonal decompositions of surfaces with isolated rational singularities from a kind of semiorthogonal decompositions of resolutions of singularities:
 \begin{thm}[{\cite{MR4382477}}]
  Let $X$ be a surface with rational singularities, and $\pi \colon Y \to X$ be a resolution. Let \[
   \Der ^{\bd} (Y) = \langle \tilde{\mathcal{A}}_1,\ldots,\tilde{\mathcal{A}}_n \rangle
  \] be a semiorthogonal decomposition of $\Der ^{\bd} (Y)$ which is compatible with $\pi$, in the sense of Definition 4.1. Then, $\mathcal{A}_i := \rpi \tilde{\mathcal{A}}_i$ for $i \in \{ 1,\ldots,r \}$ form a semiorthogonal decomposition \[
   \Der ^{\bd} (X) = \langle \mathcal{A}_1,\ldots,\mathcal{A}_n \rangle.
  \]
 \end{thm}
 We prove the similar theorem for arbitrary-dimensional projective varieties:
 \begin{mthm}[Theorem 5.3]
  Let $X$ be a projective variety with isolated rational singularities, and $\pi \colon Y \to X$ be a resolution such that the dimension of fibres of $\pi$ are at most $1$. Let \[
   \Der ^{\bd} (Y) = \langle \tilde{\mathcal{A}}_1,\ldots,\tilde{\mathcal{A}}_n \rangle
  \] be a semiorthogonal decomposition of $\Der ^{\bd} (Y)$ which is compatible with $\pi$, in the sense of Definition 4.1. Then, $\mathcal{A}_i := \rpi \tilde{\mathcal{A}}_i$ for $i \in \{ 1,\ldots,r \}$ form a semiorthogonal decomposition \[
   \Der ^{\bd} (X) = \langle \mathcal{A}_1,\ldots,\mathcal{A}_n \rangle.
  \]
 \end{mthm}
 The proof of main theorem is similar to the proof of Theorem 1.1, but we construct a new example of semiorthogonal decomposition in the three dimensional case by using main theorem in Example 5.4.
 \begin{notat}
  We work over the complex number field $\mathbb{C}$, and a variety means that separated, integral scheme of finite type over $\mathbb{C}$.
  \par For each variety $X$, we denote by $\coh (X)$ the abelian category of coherent sheaves on $X$, by $\Der _{\qc} (X)$ the derived category of quasi-coherent sheaves on $X$, by $\Der ^{\bd} (X)$ the subcategory of $\Der _{\qc} (X)$ consisting of bounded complexes of coherent sheaves, and by $\Der ^- (X)$ the subcategory of $\Der _{\qc} (X)$ consisting of complexes bounded above of coherent sheaves.
 \end{notat}
 \begin{acknow}
  I would like to sincerely thank my advisor Yukinobu Toda for many helpful suggestions and useful comments. I would like to thank Yuta Hatasa for improving my expressions and Tasuki Kinjo and Kengo Maehara for gentle encouragement. Finally, I would like to thank my parents for their support troughout my study.
 \end{acknow}
 \section{Birational geometry}
 \begin{defin}[rational singularity]
  A normal variety $X$ has rational singularities if for any resolution $\pi \colon Y \to X$, $\rpi \mathcal{O}_Y \simeq \mathcal{O}_X$.
 \end{defin}
 \begin{cor}
  Let $X$ be a variety with rational singularities. Then, ${\rpi} \circ \lpi \simeq \id$.
 \end{cor}
 \begin{proof}
  Use the projection formula.
 \end{proof}
 \begin{lem}[{\cite[Lemma 2.2]{MR4382477}}, {\cite[Lemma 2.4]{MR4276320}}]
  Let $\pi \colon Y \to X$ be a proper morphism with ﬁbres of dimension at most $1$, and $k$ be an integer. Let $\mathcal{F} \in \Der ^- (Y)$ be a complex on $Y$ bounded above which satisfies that \[
   \rip \mathcal{F} \simeq 0
  \] for all $i < k$. Then,
  \begin{eqnarray*}
   \rpi \tau ^{\leq k - 2} \mathcal{F} & \simeq & 0, \\
   \rpi \tau ^{\geq k - 1} \mathcal{F} & \simeq & \rpi \mathcal{F}.
  \end{eqnarray*}
  Especially, if $\pi$ additionally satisfies that $\rpi \mathcal{O}_Y \simeq \mathcal{O}_X$, for any $k_- \leq k_+$ and $\mathcal{G} \in \Der ^{[k_-,k_+]} (X)$,
  \begin{eqnarray*}
   \rpi \tau ^{\leq k_- - 2} \lpi \mathcal{G} & \simeq & 0, \\
   \rpi \tau ^{\geq k_- - 1} \lpi \mathcal{G} & \simeq & \mathcal{G}.
  \end{eqnarray*}
 \end{lem}
 \begin{proof}
  Consider a distinguished triangle \[
   \xymatrix@C=10pt{
    \rpi \tau ^{\leq k - 2} \mathcal{F} \ar[r] & \rpi \mathcal{F} \ar[r] & \rpi \tau ^{\geq k - 1} \mathcal{F} \ar@{.>}[r] &.
    }
  \] Now, there exists a complex $\mathcal{I} \in \Der _{\qc} (Y)$ such that it is homotopically injective and termwise injective, and $\mathcal{I} \simeq \tau ^{\leq k - 2} \mathcal{F}$. Denote by $\dfl _{\mathcal{I}} ^{\cdot}$ the differential of $\mathcal{I}$. Then, the exact sequence \[
   \xymatrix@C=10pt{
    0 \ar[r] & \im \dfl _{\mathcal{I}} ^{k-3} \ar[r] & \mathcal{I}^{k-2} \ar[r] & \mathcal{I}^{k-1} \ar[r] & \mathcal{I}^{k} \ar[r] & \cdots
   }
  \] means that
  \begin{eqnarray*}
   \rip \tau ^{\leq k - 2} \mathcal{F} & \simeq & \hoi \pi _* \mathcal{I} \\
   & \simeq & \rkp \im \dfl _{\mathcal{I}} ^{k-3} \\
   & \simeq & 0
  \end{eqnarray*}
  for all $i > k - 2 + 1 = k - 1$, where the last equality holds since the dimension of fibres of $\pi$ are at most $1$.
  \par On the other hand, for any $i \leq k - 1$, by considering the cohomology long exact sequence of the distinguished triangle \[
   \xymatrix@C=10pt{
    \rpi \tau ^{\leq k - 2} \mathcal{F} \ar[r] & \rpi \mathcal{F} \ar[r] & \rpi \tau ^{\geq k - 1} \mathcal{F} \ar@{.>}[r] &,
   }
  \] we obtain that \[
   \rip \tau ^{\leq k - 2} \mathcal{F} \simeq 0
  \] since $\ripm \tau ^{\geq k - 1} \mathcal{F} \simeq \rip \mathcal{F} \simeq 0$.
  \par Thus, $\rpi \tau ^{\leq k - 2} \mathcal{F} \simeq 0$, and also $\rpi \tau ^{\geq k_- - 1} \mathcal{F} \simeq \rpi \mathcal{F}$.
 \end{proof}
 \begin{rem}
  \cite{MR4382477,MR4276320} give a different proof for the previous lemma used spectral sequences. We give more direct proof.
 \end{rem}
 \begin{cor}[{\cite[Corollary 2.3]{MR4382477}}, {\cite[Corollary 2.5]{MR4276320}}]
  Under the assumptions of Lemma 2.5, the push-forward functor \[
   \rpi \colon \Der ^{\bd} (Y) \to \Der ^{\bd} (X)
  \] is essentially surjective.
 \end{cor}
 \begin{proof}
  For each $k_- \leq k_+$ and $\mathcal{F} \in \Der ^{[k_-,k_+]} (X)$, choose $\tau ^{\geq k_- - 1} \lpi \mathcal{F} \in \Der ^{\bd} (Y)$.
 \end{proof}
 \begin{prop}[{\cite[Lemma 3.4.1]{MR2057015}}]
  Let $X$ be a normal variety which has rational singularities, and $\pi \colon Y \to X$ be a resolution such that the dimension of fibres of $\pi$ are at most $1$. Then, all fibres of $\pi$ are a point or the union of rational curves.
 \end{prop}
 \begin{proof}
  Let $F$ be a fibre of $\pi$, and assume that $F$ is not a point. Then, $\dim F = 1$. We may assume $X$ is affine. By \cite{MR0463157}[III, Corollary 11.3], $F$ is connected. Since $F$ is projective, $\Ho ^0 (F,\mathcal{O}_F) \simeq \mathbb{C}$.
  \par Consider a distinguished triangle \[
    \xymatrix{
     \mathcal{I}_F \ar[r] & \mathcal{O}_Y \ar[r] & \mathcal{O}_F \ar@{.>}[r] & \mathcal{I}_F [1],
    }
   \] and by applying $\rga (Y,-)$ and long exact sequence, we obtain the exact sequence \[
    \xymatrix{
     \Ho ^1 (Y,\mathcal{O}_Y) \ar[r] & \Ho ^1 (F,\mathcal{O}_F) \ar[r] & \Ho ^2 (Y,\mathcal{I}_F),
    }
   \] where $\mathcal{I}_F$ is the ideal sheaf of $F$. Since $\rpi \mathcal{O}_Y \simeq \mathcal{O}_X$ and $X$ is affine, $\Ho ^1 (Y,\mathcal{O}_Y) \simeq \Ho ^1 (X,\mathcal{O}_X) \simeq 0$. Since $\pi$ is flat, and $\dim F = 1$, $\rd ^2 \pi _*\mathcal{I}_F \simeq 0$. By \cite{MR0463157}[III, Proposition 8.5], \[
   \Ho ^2 (Y,\mathcal{I}_F) \simeq \Gamma (Y, \widetilde{\Ho ^2 (Y,\mathcal{I}_F)}) \simeq \Gamma (X,\rd ^2 \pi _*\mathcal{I}_F) \simeq 0.
  \]
  \par Now, consider an irreducible component $D$ in $F$. By the exact sequence \[
    \xymatrix{
     \Ho ^1 (C,\mathcal{O}_C) \ar[r] & \Ho ^1 (D,\mathcal{O}_D) \ar[r] & \Ho ^2 (C,\mathcal{I}_D),
    }
   \] we obtain that $\Ho ^1 (D,\mathcal{O}_D) \simeq 0$. By \cite{MR0463157}[IV, Example 1.3.5], we obtain that $D \simeq \mathbb{P}^1$.
 \end{proof}
 \section{Null categories}
 Let $X$ be a projective variety with isolated rational singularities, and $\pi \colon Y \to X$ be a resolution such that the dimension of fibres of $\pi$ are at most $1$. Denote by $E_1,\ldots,E_r$ the irreducible components of exceptional loci.
 \par Let \[
  \ker _{\qc} (Y) = \{\mathcal{F} \in \Der _{\qc} (Y) \mid \rpi \mathcal{F} \simeq 0\},
 \] and define $\ker ^{\dag} \rpi = {\ker _{\qc} \rpi} \cap \Der ^{\dag} (Y)$ ($\dag \in \{ b,- \}$) and $\ker ^{\co} \rpi = {\ker _{\qc} \rpi} \cap \coh (Y)$.
 \begin{rem}
  The full subcategories $\ker _{\qc} \rpi$, $\ker ^{\bd} \rpi$, $\ker ^- \rpi$ are triangulated. Moreover, by \cite[Lemma 2.1]{MR4460094}, $\ker ^{\co} \rpi$ is abelian, and the canonical inclusion $\ker ^{\co} \rpi \to \coh (Y)$ is an exact functor.
 \end{rem}
 \begin{prop}
  There is an identity \[
   \ker ^{\bd} \rpi = \langle \mathcal{O}_{E_j} (-1) \rangle _{j = 1}^r,
  \] where the right hand side is the smallest triangulated full subcategory in $\Der ^{\bd} (Y)$ containing $\mathcal{O}_{E_j} (-1)$ for all $j \in \{1,\ldots,r \}$.
 \end{prop}
 \begin{proof}
  For any $\mathcal{F} \in \ker ^{\bd} \rpi$ and $i \in \mathbb{Z}$, we have $\Ho ^i \mathcal{F} \in \ker ^{\bd} \rpi$. Since $Y$ is noetherian, we can choose a maximal element $\mathcal{G}_0$ of $\{\mathcal{G} \in \ker ^{\co} \rpi \mid \mathcal{G} \subsetneq \mathcal{F}\}$. By maximality, $\mathcal{F} /{\mathcal{G}}$ is a simple object in $\ker ^{\co} \rpi$. By \cite[Theorem 7.13]{bodzenta2020categorifying}, $\mathcal{F} / \mathcal{G}$ forms $\mathcal{O}_{E_j} (-1)$ for some $j$. By repeating, we obtain a filtration \[
   \Ho ^i \mathcal{F} = \mathcal{G}_0 \supsetneq \mathcal{G}_1 \supsetneq \mathcal{G}_2 \supsetneq \cdots
  \] with $\mathcal{G}_i / \mathcal{G}_{i+1} \simeq \mathcal{O}_{E_{j_i}} (-1)$ for each $i$. Assume this filtration is unbounded.
  \par Fix an ample divisor $H \in \cl (Y)$, and let $P_H (\mathcal{G},n) = \chi (\mathcal{G} \otimes _{\mathcal{O}_Y} \mathcal{O}_Y (nH))$ called the Hilbert polynomial of $\mathcal{G}$ with respect to $H$ for each $\mathcal{G} \in \coh (Y)$. Since $\dim \supp \mathcal{G}_i \leq 1$, $P_H (\mathcal{G}_i,n)$ forms $a_i n + b_i$ for some $a_i,b_i \in \mathbb{Z}$. Since $H$ is ample, $P_H (\mathcal{G}_i,n) = \dim \Ho ^0 (X,\mathcal{G}_i \otimes _{\mathcal{O}_Y} \mathcal{O}_Y (nH)) \geq 0$ for large $n$, and we obtain $a_i \geq 0$. Similarly, we have $P_H (\mathcal{O}_{E_{j_i}} (-1),n) = (H \cdot E_{j_i}) n$. Note that $(H \cdot E_{j_i}) > 0$ by Nakai-Moishezon criterion. However, by unboundedness, $a_i < 0$ for large $i$. This leads to contradiction.
  \par Thus, the filtration above is bounded. This implies $\Ho ^i \mathcal{F} \in \ker ^{\co} \rpi$ and $\mathcal{F} \in \ker ^{\bd} \rpi$.
 \end{proof}
 \section{Semiorthogonal decompositions}
 \begin{defin}
  A semiorthogonal decomposition $\Der ^{\bd} (Y) = \langle \tilde{\mathcal{A}}_1,\ldots,\tilde{\mathcal{A}}_n \rangle$ is called compatible with $\pi$ if for each irreducible component $E_j$ of the exceptional loci of $\pi$, there exists $i \in \{ 1,\ldots,n \}$ such that $\mathcal{O}_{E_j} (-1) \in \tilde{\mathcal{A}}_i$.
 \end{defin}
 Define $\mathcal{E}_i = \{E_j \mid \mathcal{O}_{E_j}(-1) \in \tilde{\mathcal{A}}_i\}$ and $D_i = \bigcup _{E_j \in \mathcal{E}_i} E_j$ for each $i \in \{ 1,\ldots,n \}$.
 \begin{cor}
  For any $i \neq i'$, $D_i \cap D_{i'} = \emptyset$.
 \end{cor}
 \begin{proof}
  We may assume $i > i'$. For each $E \in \mathcal{E}_i$ and $E' \in \mathcal{E}_{i'}$, \[
   \rHom _Y (\mathcal{O}_E (-1), \mathcal{O}_{E'} (-1)) \simeq 0
  \] by semiorthogonality. Assume $E \cap E' \neq \emptyset$. We have
  \begin{eqnarray*}
   \rHom _Y (\mathcal{O}_E (-1),\mathcal{O}_{E'} (-1)) & \simeq & \rgam (Y,\rHomu _Y (\mathcal{O}_E (-1),\mathcal{O}_{E'} (-1))) \\
   & \simeq & \rgam (Y,(\mathcal{O}_E (-1))^{\vee} \otimes ^{\lm}_Y \mathcal{O}_{E'} (-1)).
  \end{eqnarray*}
  Moreover, we can compute that $(\mathcal{O}_E (-1))^{\vee} \simeq \mathcal{O}_E (-1) \otimes _Y \omega _Y [-2]$. In fact, for any $\mathcal{G} \in \Der ^{\bd} (Y)$,
  \begin{eqnarray*}
   \rHom _Y (\mathcal{G},(\mathcal{O}_E (-1))^{\vee}) & \simeq & \rHom _Y (\mathcal{G} \otimes ^{\lm}_Y \mathcal{O}_E (-1),\mathcal{O}_Y) \\
   & \simeq & \rHom _Y (i_* (\li ^* \mathcal{G} \otimes _E \mathcal{O}_E (-1)),\mathcal{O}_Y) \\
   & \simeq & \rHom _E (\li ^* \mathcal{G} \otimes _E \mathcal{O}_E (-1),\omega _E \otimes _E i^* \omega _Y^{\vee} [-2]) \\
   & \simeq & \rHom _E (\li ^* \mathcal{G},\mathcal{O}_E (1) \otimes _E \omega _E \otimes _E i^* \omega _Y^{\vee} [-2]) \\
   & \simeq & \rHom _Y (\mathcal{G},i_* (\mathcal{O}_E (1) \otimes _E \omega _E \otimes _E i^* \omega _Y^{\vee}) [-2]) \\
   & \simeq & \rHom _Y (\mathcal{G},i_* (\mathcal{O}_E (1) \otimes _E \omega _E) \otimes _Y \omega _Y^{\vee} [-2]),
  \end{eqnarray*}
  functorially, where $i \colon E \hookrightarrow Y$ is the inclusion morphism. Note that the second and the last equalities hold by the projection formula, and the fifth equality holds by Grothendieck duality. Hence, by Yoneda lemma, \[
   (\mathcal{O}_E (-1))^{\vee} \simeq \mathcal{O}_E (1) \otimes _E \omega _E \otimes _Y \omega _Y^{\vee} [-2] \simeq \mathcal{O}_E (-1) \otimes _Y \omega _Y^{\vee} [-2],
  \] where the last equality holds since $E$ is a rational curve.
  \par Now, we consider a spectral sequence
  \begin{eqnarray*}
   E_2^{p,q} & = & \Ho ^p (Y,\Ho ^q (\mathcal{O}_E (-1) \otimes ^{\lm}_Y \omega _Y^{\vee} \otimes ^{\lm}_Y \mathcal{O}_{E'} (-1) [-2])) \\
   & \Rightarrow & \Ho ^{p+q} (Y,\mathcal{O}_E (-1) \otimes ^{\lm}_Y \omega _Y^{\vee} \otimes ^{\lm}_Y \mathcal{O}_{E'} (-1) [-2]).
  \end{eqnarray*}
  Since $j_x^* \Ho ^0 (\mathcal{O}_E (-1) \otimes ^{\lm}_Y \omega _Y^{\vee} \otimes ^{\lm}_Y \mathcal{O}_{E'} (-1)) \simeq k(x) \neq 0$ for any $x \in E \cap E'$, by Nakayama's lemma, $(\Ho ^0 (\mathcal{O}_E (-1) \otimes ^{\lm}_Y \omega _Y^{\vee} \otimes ^{\lm}_Y \mathcal{O}_{E'} (-1)))|_x \neq 0$, where $j_x \colon x \hookrightarrow Y$ is the inclusion morphism. Since $\dim E \cap E' = 0$, we obtain that
  \begin{eqnarray*}
   E_2^{0,2} & = & \Ho ^0 (X,\Ho ^0 (\mathcal{O}_E (-1) \otimes ^{\lm}_Y \omega _Y^{\vee} \otimes ^{\lm}_Y \mathcal{O}_{E'} (-1))) \\
   & \simeq & \Ho ^0 (E \cap E',\Ho ^0 (\mathcal{O}_E (-1) \otimes ^{\lm}_Y \omega _Y^{\vee} \otimes ^{\lm}_Y \mathcal{O}_{E'} (-1))) \\
   & \simeq & \bigoplus _{x \in E \cap E'} (\Ho ^0 (\mathcal{O}_E (-1) \otimes ^{\lm}_Y \omega _Y^{\vee} \otimes ^{\lm}_Y \mathcal{O}_{E'} (-1)))|_x \\
   & \neq & 0.
  \end{eqnarray*}
  Hence, $E_{\infty}^{0,2} \simeq E_2^{0,2} \neq 0$ since $E_2^{p,q} \simeq 0$ if $p > 0$ or $q > 2$. Therefore, \[
   \Ho ^{2} (Y,\mathcal{O}_E (-1) \otimes ^{\lm}_Y \omega _Y^{\vee} \otimes ^{\lm}_Y \mathcal{O}_{E'} (-1) [-2]) \neq 0.
  \] This leads to contradiction. Thus, $E \cap E' = \emptyset$, and hence, $D_i \cap D_{i'} = \emptyset$.
 \end{proof}
 Let \[
  \Der ^{\bd} (Y) = \langle \tilde{\mathcal{A}}_1,\ldots,\tilde{\mathcal{A}}_n \rangle
 \] be a semiorthogonal decomposition of $\Der ^{\bd} (Y)$ which is compatible with $\pi$.
 \begin{lem}
  For each $\mathcal{F} \in \ker ^{\bd} \rpi$, we can write $\mathcal{F} \simeq \bigoplus _{i = 1}^n \mathcal{F}_i$ with $\mathcal{F}_i \in \tilde{\mathcal{A}}_i \cap \ker ^{\bd} \rpi$.
 \end{lem}
 \begin{proof}
  By proposition 3.2, the null category $\ker ^{\bd} \rpi = \langle \mathcal{O}_{E_j}(-1) \rangle _{j = 1}^r$ consists of complexes obtained by repeating a finite number of extensions and shifts of $\mathcal{O}_{E_j} (-1)$ for $j \in \{1,\ldots,r \}$. We prove the lemma by induction on the structure of complexes. We take $\mathcal{F} \in \ker ^{\bd} \rpi$.
   \begin{itemize}
    \item[-] If $\mathcal{F}$ forms $\mathcal{O}_{E_j} (-1)$ with $E_j \in \mathcal{E}_{i_0}$, we can choose \[
     \mathcal{F}_i =
      \begin{cases}
       \mathcal{O}_{E_j} (-1) & i = i_0, \\
       0 & otherwise.
      \end{cases}
     \]
    \item[-] If $\mathcal{F}$ forms $\cone(\varphi \colon \bigoplus _{i = 1}^n \mathcal{G}_i \to \bigoplus _{i = 1}^n \mathcal{H}_i)$, write $\varphi = (\varphi _{ij} \colon \mathcal{G}_i \to \mathcal{H}_j)_{i,j = 1}^n$. By Corollary 4.2, $\varphi _{ij} = 0$ for $i \neq j$. Hence, $\varphi = \bigoplus _{i= 1}^n \varphi _{ii}$, and we can take $\mathcal{F}_i = \cone(\varphi _{ii})$.
    \item[-] If $\mathcal{F}$ forms ($\bigoplus _{i = 1}^n \mathcal{G}_i)[\pm 1]$, we can take $\mathcal{F}_i = \mathcal{G}_i [\pm 1]$. \qedhere
   \end{itemize}
 \end{proof}
 \begin{rem}
  Since the author do not be sure the correctness of \cite[Lemma 2.15]{MR4382477}, we restrict the previous lemma to $\Der ^{\bd} (Y)$ instead of $\Der ^- (Y)$. Instead, we work the following section in $\Der ^{\bd} (Y)$, by applying the truncation functor to unbounded complexes.
 \end{rem}
 \section{The proof of the main theorem}
  Define $\mathcal{A}_i$ as a full subcategory $\rpi \tilde{\mathcal{A}}_i$ in $\Der ^{\bd} (X)$ for each $i \in \{ 1,\ldots,n \}$.
 \begin{lem}
  For each $i \in \{ 1,\ldots,n \}$ and $k_- \leq k_+$, \[
     \tau ^{\geq l} \lpi \rpi (\tilde{\mathcal{A}}_i \cap \Der ^{[k_-,k_+]} (Y)) \subseteq \langle \tilde{\mathcal{A}}_i,\tilde{\mathcal{A}}_{i+1} \cap \ker ^{\bd} \rpi,\ldots,\tilde{\mathcal{A}}_n \cap \ker ^{\bd} \rpi \rangle
    \] for all $l \leq k_- - 1$.
 \end{lem}
 \begin{proof}
  For each $\mathcal{F} \in \tilde{\mathcal{A}}_i \cap \Der ^{[k_-,k_+]} (Y)$, we can take a distinguished triangle \[
     \xymatrix@C=10pt{
      \lpi \rpi \mathcal{F} \ar[r] & \mathcal{F} \ar[r] & \mathcal{G} \ar@{.>}[r] &.
     }
    \] Since ${\rpi} \circ \lpi \simeq \id$, $\mathcal{G} \in \ker ^- \rpi$. Moreover, by Lemma 4.3, for each $l \leq k_- - 1$, there is the decomposition \[
     \tau ^{\geq l} \mathcal{G} \simeq \bigoplus _{j = 1}^n \mathcal{G}_{lj}
    \] with $\mathcal{G}_{lj} \in \tilde{\mathcal{A}} _i \cap \ker ^{\bd} \rpi$. Then, the truncation morphism $(\tau ^{\geq l} \lpi \mathcal{F} \to \tau ^{\geq l + 1} \lpi \mathcal{F}) = \bigoplus _{j = 1} ^n (\mathcal{G}_{lj} \to \mathcal{G}_{(l + 1)j})$ since $\supp \mathcal{G}_{lj} \cap \supp \mathcal{G}_{(l + 1)j'} = \emptyset$ for any $j \neq j'$. Thus, \[
     \mathcal{G} \simeq \bigoplus _{j = 1} ^n \holim _{l \leq k_- - 1} \mathcal{G}_{lj}.
    \] Define $\mathcal{G}_j = \holim _{l \leq k_- - 1} \mathcal{G}_{lj}$. Note that $\mathcal{G}_j \in \ker ^- \rpi$ by the previous direct sum decomposition of $\mathcal{G} \in \ker ^- \rpi$. For any $j < i$, by semiorthogonality, \[
     \rHom _Y (\mathcal{F},\holim _{l \leq k_- - 1} \mathcal{G}_{lj}) \simeq \holim _{l \leq k_- - 1} \rHom _Y (\mathcal{F},\mathcal{G}_{lj}) \simeq 0.
    \] Hence, \[
     \rHom _Y (\mathcal{F},\mathcal{G}_j) \simeq 0.
    \] Moreover, \[
     \rHom _Y (\lpi \rpi \mathcal{F},\mathcal{G}_j) \simeq \rHom _X (\rpi \mathcal{F},\rpi \mathcal{G}_j) \simeq \rHom _X (\rpi \mathcal{F},0) \simeq 0.
    \] Therefore, $\rHom _Y (\mathcal{G},\mathcal{G}_j) \simeq 0$, and hence, $\mathcal{G}_j \simeq 0$.
    \par For all $l \leq k_- - 1$ and $p \geq l$, $\Ho ^p \mathcal{G}_j \to \Ho ^p \mathcal{G}_{lj}$ is isomorphic since the morphism $(\Ho ^p \mathcal{G} \to \Ho ^p \tau ^{\geq l} \mathcal{G}) = \bigoplus _{i = 1} ^n (\Ho ^p \mathcal{G}_j \to \Ho ^p \mathcal{G}_{lj})$ is isomorphic. Hence, $\Ho ^p \mathcal{G}_{lj} \simeq 0$. That is, $\mathcal{G}_{lj} \simeq 0$. Hence, \[
     \tau ^{\geq l} \mathcal{G} \in \langle \tilde{\mathcal{A}} _i \cap \ker ^{\bd} \rpi,\tilde{\mathcal{A}} _{i+1} \cap \ker ^{\bd} \rpi,\ldots,\tilde{\mathcal{A}} _n \cap \ker ^{\bd} \rpi \rangle.
    \] Since $\tau ^{\geq l} \mathcal{F} \simeq \mathcal{F} \in \tilde{\mathcal{A}}_i$, \[
     \tau ^{\geq l} \lpi \rpi \mathcal{F} \in \langle \tilde{\mathcal{A}}_i,\tilde{\mathcal{A}}_{i+1} \cap \ker ^{\bd} \rpi,\ldots,\tilde{\mathcal{A}}_n \cap \ker ^{\bd} \rpi \rangle. \qedhere
    \]
 \end{proof}
 \begin{prop}
  For all $i \in \{1,\ldots,n \}$, $\mathcal{A}_i$ is a triangulated full subcategory in $\Der ^{\bd} (X)$.
 \end{prop}
 \begin{proof}
  For any $k_- \leq k_+$ and $f \colon \rpi \mathcal{F} \to \rpi \mathcal{G}$ in $\rpi (\tilde{\mathcal{A}}_i \cap \Der ^{[k_-,k_+]} (Y))$, $f = \rpi \tau ^{\geq k_- - 1} \lpi f$ by Lemma 2.3. Hence, \[
   \cone f \simeq \cone \rpi \tau ^{\geq k_- - 1} \lpi f \simeq \rpi \cone \tau ^{\geq k_- - 1} \lpi f.
  \] Moreover, by Lemma 5.1, \[
   \rpi \cone \tau ^{\geq k_- - 1} \lpi f \in \mathcal{A}_i. \qedhere
  \]
 \end{proof}
 Let $\tilde{\alpha}_i \colon \Der ^{\bd} (Y) \to \tilde{\mathcal{A}}_i$ be the projections.
 \begin{thm}
  There is a semiorthogonal decomposition \[
   \Der ^{\bd} (X) = \langle \mathcal{A}_1,\ldots,\mathcal{A}_n \rangle
  \] such that the projections $\alpha _i$ are isomorphic to ${\rpi} \circ \tilde{\alpha}_i \circ \tau^{\geq k_- - 1} \circ \lpi$.
 \end{thm}
 \begin{proof}
    First, we prove the semiorthogonality. Take $\mathcal{F} \in \rpi (\tilde{\mathcal{A}}_i \cap \Der ^{[k_-,k_+]} (Y))$ and $\mathcal{G} \in \tilde{\mathcal{A}}_j$ for each $i > j$ and $k_- \leq k_+$. Then, by Lemma 5.1, for all $l \leq k_- - 1$, \[
     \tau ^{\geq l} \lpi \mathcal{F} \in \langle \tilde{\mathcal{A}}_i,\tilde{\mathcal{A}}_{i+1} \cap \ker ^{\bd} \rpi,\ldots,\tilde{\mathcal{A}}_n \cap \ker ^{\bd} \rpi \rangle.
    \] By a distinguished triangle \[
     \xymatrix@C=10pt{
      \Ho ^l \lpi \mathcal{F} \ar[r] & \tau ^{\geq l} \lpi \mathcal{F} \ar[r] & \tau ^{\geq l + 1} \lpi \mathcal{F} \ar@{.>}[r] &,
     }
    \] we obtain \[
     \Ho ^l \lpi \mathcal{F} \in \langle \tilde{\mathcal{A}}_i,\tilde{\mathcal{A}}_{i+1} \cap \ker ^{\bd} \rpi,\ldots,\tilde{\mathcal{A}}_n \cap \ker ^{\bd} \rpi \rangle.
    \] Therefore, by semiorthogonality, \[
     \rHom _Y (\tau ^{\geq k_- - 1} \lpi \mathcal{F},\mathcal{G}) \simeq 0,
    \] and \[
     \rHom _Y (\Ho ^l \lpi \mathcal{F},\mathcal{G}) \simeq 0.
     \] Hence, By considering a spectral sequence \[
      E_2 ^{p,q} = \Ext _Y ^p (\Ho ^{-q} \lpi \mathcal{F},\mathcal{G}) \Rightarrow \Ext _Y ^{p + q} (\tau ^{\leq k_- - 2} \lpi \mathcal{F},\mathcal{G}),
     \] we obtain that \[
      \rHom _Y (\tau ^{\leq k_- - 2} \lpi \mathcal{F},\mathcal{G}) \simeq 0.
     \] Thus, \[
     \rHom _X (\mathcal{F},\rpi \mathcal{G}) \simeq \rHom _Y (\lpi \mathcal{F},\mathcal{G}) \simeq 0.
    \]
    \par Next, for each $\mathcal{F} \in \Der ^{\bd} (X)$, take the decomposition of $\tilde{\mathcal{F}} := \tau ^{\geq k_- - 1} \lpi \mathcal{F}$: \[
     \xymatrix@!C=10pt{
      0 \ar[rr] & & \tilde{\mathcal{F}}_n \ar[r] \ar[ld] & \cdots \ar[r] & \tilde{\mathcal{F}}_3 \ar[rr] & & \tilde{\mathcal{F}}_2 \ar[rr] \ar[ld] & & \tilde{\mathcal{F}}. \ar[ld] \\
      & \tilde{\alpha}_n \tilde{\mathcal{F}} \ar@{.>}[lu] & & & & \tilde{\alpha}_2 \tilde{\mathcal{F}} \ar@{.>}[lu] & & \tilde{\alpha}_1 \tilde{\mathcal{F}} \ar@{.>}[lu]
     }
    \] By applying $\rpi$, we obtain the following diagram: \[
     \xymatrix@!C=10pt{
      0 \ar[rr] & & \rpi \tilde{\mathcal{F}}_n \ar[r] \ar[ld] & \cdots \ar[r] & \rpi \tilde{\mathcal{F}}_3 \ar[rr] & & \rpi \tilde{\mathcal{F}}_2 \ar[rr] \ar[ld] & & \mathcal{F}. \ar[ld] \\
      & \rpi \tilde{\alpha}_n \tilde{\mathcal{F}} \ar@{.>}[lu] & & & & \rpi \tilde{\alpha}_2 \tilde{\mathcal{F}} \ar@{.>}[lu] & & \rpi \tilde{\alpha}_1 \tilde{\mathcal{F}} \ar@{.>}[lu]
     }
    \] Since $\rpi \tilde{\alpha}_i \tilde{\mathcal{F}} \in \mathcal{A}^- _i$ for all $i \in \{ 1,\ldots,n \}$, the previous diagram proves the required semiorthogonal decomposition is well-defined. Moreover, the projections $\alpha _i \simeq {\rpi} \circ \tilde{\alpha}_i \circ \tau^{\geq k_- - 1} \circ \lpi$.
 \end{proof}
 \begin{exe}
  Let $X = (xy-zw = 0) \subseteq \mathbb{P}^4$ be a conifold. By blowing-up along a closed subscheme $Z = (x = z = 0) \subseteq X$, we obtain a small resolution $\pi \colon Y := \bl _Z X \to X$. We can check that $Y \simeq \mathbb{P} _{\mathbb{P}^1} (\mathcal{O} (-1) \oplus \mathcal{O} (-1) \oplus \mathcal{O})$, and the exceptional curve $C$ of $\pi$ corresponds to $\mathbb{P} _{\mathbb{P}^1} (0 \oplus 0 \oplus \mathcal{O}) \subseteq \mathbb{P} _{\mathbb{P}^1} (\mathcal{O} (-1) \oplus \mathcal{O} (-1) \oplus \mathcal{O})$. By the projection bundle formula, we have a semiorthogonal decomposition \[
   \Der ^{\bd} (Y) = \langle \mathcal{O} _Y (-2E),\mathcal{O} _Y (-2E + H),\mathcal{O} _Y (-E - H),\mathcal{O} _Y (-E),\mathcal{O} _Y (-H),\mathcal{O} _Y \rangle,
  \] where $E$ is the divisor corresponding to $\mathbb{P} _{\mathbb{P}^1} (\mathcal{O} (-1) \oplus \mathcal{O} (-1) \oplus 0)$, and $H$ is the pullback of the point class from $\mathbb{P}^1$. By mutation, we obtain that \[
   \Der ^{\bd} (Y) = \langle \mathcal{O} _Y (-2E),\mathcal{L},\mathcal{O} _Y (-2E + H),\mathcal{O} _Y (-E),\mathcal{O} _Y (-H),\mathcal{O} _Y \rangle,
  \] where the left mutation $\mathcal{L} := \lm _{\mathcal{O} (-2E + H)} \mathcal{O} (-E - H)$ is defined as the $(-1)$-shift of the mapping cone of the evaluation morphism \[
   \rHom _Y (\mathcal{O} (-2E + H),\mathcal{O} (-E - H)) \otimes _{\mathcal{C}} \mathcal{O} (-2E + H) \to \mathcal{O} (-E - H).
  \] Note that \[
   \lm _{\mathcal{O} (-2E + H)} \mathcal{O} (-E - H) \simeq \mathcal{E}[-1],
  \] where $\mathcal{E}$ is a vector bundle of rank $2$ which is obtained as the non-split extension \[
   \xymatrix@C=10pt{
    0 \ar[r] & \mathcal{O} (-E - H) \ar[r] & \mathcal{E} \ar[r] & \mathcal{O} (-2E + H) \ar[r] & 0.
   }
  \] Prove that \[
   \mathcal{O} _C (-1) \in \langle \mathcal{O} _Y (-2E + H),\mathcal{O} _Y (-E),\mathcal{O} _Y (-H) \rangle.
  \]
  \par Let $D_1,D_2 \in \cl (Y)$ be the divisors corresponding to surfaces $\mathbb{P} _{\mathbb{P}^1} (\mathcal{O} (-1) \oplus 0 \oplus \mathcal{O}),\mathbb{P} _{\mathbb{P}^1} (0 \oplus \mathcal{O} (-1) \oplus \mathcal{O})$ of $Y$, respectively. Then, the intersection of $D_1$ and $D_2$ is equal to $C$. Since $\cl (Y) \simeq \mathbb{Z} E \oplus \mathbb{Z} H$, the divisor $D_1$ is represented by the form $aE + bH$. Now, $C$ is a $(-1)$-curve of $D_1$. In fact, since resolved conifold $\tot _{\mathbb{P}^1} (\mathcal{O} (-1) \oplus \mathcal{O} (-1))$ is Calabi-Yau, $Y$ is Calabi-Yau in a neighborhood of $C$, that is, $K_Y \cdot C = 0$. Also, since the self intersection number of $C$ in $D_1$ is $-1$, and the genus of the rational curve $C$ is 0, $K_{D_1} \cdot C = -1$. Therefore,
  \begin{eqnarray*}
   -1 & = & K_{D_1} \cdot C \\
   & = & K_{D_1} | _C \\
   & = & ((K_Y + D_1) | _{D_1}) | _C \\
   & = & (K_Y + D_1) | _C \\
   & = & K_Y | _C + D_1 | _C \\
   & = & 0 + D_1 \cdot C \\
   & = & D_1 \cdot C.
  \end{eqnarray*}
  Moreover, we can easily check that $E \cdot C = 0$ and $H \cdot C = 1$. Hence, \[
   -1 = D_1 \cdot C = (aE + bH) \cdot C = b.
  \] Let $L$ be a line in the pullback of the point from $\mathbb{P}^1$ in $Y$. Then, $D_1 \cdot L = 1$, $E \cdot L = 1$, and $H \cdot L = 0$. Hence, \[
   1 = D_1 \cdot L = (aE + bH) \cdot L = a.
  \] That is, $D_1 = E - H$. Similarly, we can check that $D_2 = E - H$.
  \par Consider the Koszul resolution of $\mathcal{O}_C$ and twist by $\mathcal{O} (-H)$, we obtain a exact sequence \[
   \xymatrix@C=10pt{
    0 \ar[r] & \mathcal{O} (H - 2E) \ar[r] & \mathcal{O} (-E) ^{\oplus 2} \ar[r] & \mathcal{O} (-H) \ar[r] & \mathcal{O}_C (-1) \ar[r] & 0.
   }
  \] This means that \[
   \mathcal{O} _C (-1) \in \langle \mathcal{O} _Y (-2E + H),\mathcal{O} _Y (-E),\mathcal{O} _Y (-H) \rangle.
  \] Thus, by taking that
  \begin{eqnarray*}
   \tilde{\mathcal{A}}_1 & = & \langle \mathcal{O} (-2E) \rangle, \\
   \tilde{\mathcal{A}}_2 & = & \langle \mathcal{E} \rangle, \\
   \tilde{\mathcal{A}}_3 & = & \langle \mathcal{O} (-2E + H),\mathcal{O} (-2E),\mathcal{O} (-H) \rangle, \\
   \tilde{\mathcal{A}}_4 & = & \langle \mathcal{O} \rangle,
  \end{eqnarray*}
  we obtain a semiorthogonal decomposition \[
   \Der ^{\bd} (Y) = \langle \tilde{\mathcal{A}}_1,\tilde{\mathcal{A}}_2,\tilde{\mathcal{A}}_3,\tilde{\mathcal{A}}_4 \rangle,
  \] Therefore, by Theorem 5.3, we obtain a nontrivial semiorthogonal decomposition \[
   \Der ^{\bd} (X) = \langle \mathcal{A}_1,\mathcal{A}_2,\mathcal{A}_3,\mathcal{A}_4 \rangle,
  \] such that
  \begin{eqnarray*}
   \mathcal{A}_1 & = & \langle \rpi \mathcal{O} (-2E) \rangle, \\
   \mathcal{A}_2 & = & \langle \rpi \mathcal{E} \rangle, \\
   \mathcal{A}_3 & = & \langle \rpi \mathcal{O} (-2E + H),\rpi \mathcal{O} (-E),\rpi \mathcal{O} (-H) \rangle, \\
   \mathcal{A}_4 & = & \langle \mathcal{O} _X \rangle.
  \end{eqnarray*}
 \end{exe}
 \bibliography{ref}
 \bibliographystyle{amsplain}
\end{document}